\documentclass[12pt]{amsart}
\usepackage{amsmath,amsthm,latexsym,amscd,amsbsy,amssymb,url}
\usepackage[all]{xypic}
 \relax

\chardef\bslash=`\\ 

\makeatletter
\def\verbatim{\interlinepenalty\@M \@verbatim
  \leftskip\@totalleftmargin\advance\leftskip2pc
  \frenchspacing\@vobeyspaces \@xverbatim}
\makeatother
\newtheorem{thm}{Theorem}[section]
\newtheorem{cor}[thm]{Corollary}
\newtheorem{lem}[thm]{Lemma}
\newtheorem{pro}[thm]{Proposition}
\newtheorem{defin}[thm]{Definition}
\newtheorem{rem}[thm]{Remark}
\newtheorem{ex}[thm]{Example}
\newtheorem{que}[thm]{Question}
\newtheorem{problem}[thm]{Problem}

\begin{document}

\centerline{\it In memory of my dearest friend Natashka Gamzina-Kandaurova}
\centerline{\it May your light guide me to kindness}

\title
{A Glance into the Anatomy of Monotonic Maps}
\author{Raushan  Buzyakova}
\email{Raushan\_Buzyakova@yahoo.com}



\begin{abstract}{
Given an autohomeomorphism on an ordered topological space  or its subspace, we show that  it is sometimes possible to introduce a new topology-compatible order on that space so that the same map is monotonic with  respect to the new ordering. We note that the  existence of such a re-ordering for a given map is equivalent to  the map being conjugate (topologically equivalent) to a monotonic map on some  homeomorphic ordered space. We observe that the latter cannot  always be chosen to be order-isomorphic to the original space.
Also, we identify other routes that may lead to similar affirmative statements for other classes of spaces and maps.

}
\end{abstract}

\maketitle
\markboth{R. Buzyakova}{A Glance into the Anatomy of Monotonic Maps}
{ }

\section{Introduction}\label{S:introduction}
\par\bigskip
It is one of classical problems of various areas of topology if a given continuous map on a topological space with perhaps a richer structure has nice properties related to this rich structure. For clarity of exposition, let us agree on terminology. An {\it autohomeomorphism} on a topological space $X$ is any homeomorphism of $X$ onto itself. An open interval with end points $a$ and $b$  of a linearly ordered set $L$ will be denoted by $(a,b)_L$. If it is clear what ordered set is under consideration, we simply write $(a,b)$. The same concerns other types of intervals. Linearly ordered topological spaces are abbreviated as LOTS and their subspaces as GO-spaces. We will mostly be concerned with GO-spaces. It is due to \v Cech (\cite{BL})  that a Hausdorff  space $X$  is a GO-space if and only if a family of convex sets with respect to some ordering on $X$ is a basis for the topology of  $X$. Given a GO-space $X$, an order $\prec$ on $X$ is said to be {\it GO-compatible} if some collection of $\prec$-convex subsets of $\langle X, \prec\rangle$ is a basis for the topology of $X$. Note that if $X$ is a LOTS, a GO-compatible order on $X$ need not witness the fact that $X$ is a LOTS. We will be concerned with the following general problem.
\par\bigskip\noindent
\begin{problem}\label{problem:main}
Let $X$ be a GO-space and let $f$ be an autohomeomorphism on $X$. What conditions on $X$ and/or $f$ guarantee that  $X$ has a GO-compatible ordering with respect to which  $f$ is monotonic? 
\end{problem}

\par\bigskip\noindent
Since monotonicity is an order-dependent concept, we will specify with respect to which ordering a map is monotonic. If no clarification is given, the assumed order is the original one and should be clear from the context.
Since our discussion will be around Problem \ref{problem:main}, we will isolate the target property into a definition.

\par\bigskip\noindent
\begin{defin}\label{def:potmon}
An autohomeomorphism  $f$ on a GO-space is potentially monotonic if there exists a GO-compatible order on $X$  with respect to which $f$ is monotonic.
\end{defin}

\par\bigskip\noindent
Definition \ref{def:potmon}  is equivalent to the following definition:

\par\bigskip\noindent
\begin{defin}\label{def:potmon1}
(Equivalent to \ref{def:potmon})
An autohomeomorphism $f$ on a GO-space is potentially monotonic if there exists a GO-space  $Y$,  a homeomorphism $h:X\to Y$, and a monotonic autohomeomorphism  $m$ on $Y$ such that $f=h^{-1}\circ m\circ h$.
\end{defin}

\par\bigskip\noindent
To see why these two definitions are equivalent, let $f$ be an autohomeomorphism on a GO-space $X$. Assume $f$ is potentially monotonic by Definition \ref{def:potmon}. Fix a GO-compatible order $\prec$ on $X$ with respect to which $X$ is monotonic. Put $Y = \langle X, \prec\rangle$, $h = id_X$ (the identity map), and $m=f$. Clearly, $f = id_X^{-1}\circ m\circ id_X$. Hence, $f$ is potentially monotonic with respect to Definition \ref{def:potmon1}. We now assume that $f$ is potentially monotonic with respect to Definition \ref{def:potmon1}. Fix $Y,f,m$ as in the definition. The order on $Y$ induces an order $\prec$  on $X$ as follows: $a\prec b$ if and only if $h(a)<h(b)$. Since $h$ is a homeomorphism, $\prec$ is compatible with the GO-topology of $X$. Next,  let us show that $f$ is $\prec$-monotonic. We have $a\prec b$ is equivalent to $h(a)<h(b)$.  By the choice of $m$, the latter is equivalent to $m\circ h(a) < m\circ h(b)$. By the definition of $\prec$, the latter is equivalent to 
 $h^{-1}\circ m\circ h(a) \prec  h^{-1}\circ m\circ h(b)$. Since $f = h^{-1}\circ m\circ h$, we conclude that $f(a)\prec f(b)$.

\par\bigskip\noindent
One may wonder if the property in Definition \ref{def:potmon}  is equivalent to the property of being topologically equivalent to a monotonic map with respect to the existing order. Recall that homeomorphisms $f,g:X\to X$ are {\it topologically equivalent} (or conjugate)if there exists a homeomorphism $t:X\to X$ such that $t\circ f = g\circ t$. A supported explanation will be given later in Remark \ref{rem:rem1} that a map can be  potentially monotonic but not  topologically equivalent to a monotonic map (with respect to the existing order). It is clear, however, from Definition \ref{def:potmon1}   that a map  topologically equivalent to a monotonic map is potentially monotonic.
In our arguments, given a monotonic function $f$ on a GO-space $L$ and an $x\in L$,  we will make a frequent  use of the set $\{f^n(x): n\in \mathbb Z\}$. In literature,  similarly defined sets are often referred to as the orbit of $x$ under $f$. We will also refer to this set as {\it the $f$-orbit of $x$}. Similarly, the {\it $f$-orbit} of  a set $A\subset X$ is the collection $\{f^n(A): n\in \mathbb Z\}$.
By looking at the behavior of monotonic maps on the reals, we quickly observe that the orbit of each point under such maps exhibits very strong properties. Namely, the following holds.

\par\bigskip\noindent
\begin{pro}\label{pro:strongdiscreteorbit}
Let $f$ be a fixed-point free monotonic autohomeomorphism on a GO-space $L$ and $x\in L$. Then there exists an open neighborhood $I$ of $x$ such that $f^n(I)\cap  f^m(I)=\emptyset$ for any distinct integers $n$ and $m$ and  $\{f^n(I): n \in \mathbb Z\}$ is a discrete family.
\end{pro}
\begin{proof}
Without loss of generality, we may assume that $f$ is strictly increasing. Fix $x\in L$. We have two cases.
\begin{description}
	\item[\rm Case ($x$ is isolated)] Let us show that  $I = \{x\}$ is as desired. By strict monotonicity, $f^n(x)\not = f^m(x)$ for distinct integers $n$ and $m$. To show that $S=\{f^n(\{x\}): n\in \mathbb Z\}$ is a discrete family of sets, fix $y\in L$. If $y=f^n(x)$ for some $n$, then $\{y\}$ is an open neighborhood of $y$ that meets exactly one  element of the collection, namely, $f^n({x})$. Assume $y\not = f^n(x)$ for any $n$ and $y$ is a limit point for $ S$. By monotonicity, $\lim\limits_{n\to \infty}f^n(x)=y$ or $\lim\limits_{n\to \infty}f^{-n}(x)=y$. By continuity, $f(y)=y$, contradicting the fact that $f$ is fixed-point free.
	\item[\rm Case ($x$ is not isolated)] Since $f$ is an increasing homeomorphism, the intervals  $ (x, f(x))$ and $(f^{-1}(x),x)$  are  not empty. Pick and fix $a\in (f^{-1}(x), x)$. Since $f$ is strictly increasing, $f(a)\in (x, f(x))$.  Let us show that  $I=(a,f(a))$ is as desired. First due to monotonicity, $f^n(I) = (f^n(a), f^{n+1}(a))$. Therefore, $f^n(I)$ misses $f^m(I)$ for distinct $n,m\in \mathbb Z$. The fact that $\{f^n(I): n\in \mathbb Z\}$  is a discrete collection is proved as in the previous case.
\end{description}
\end{proof}

\par\bigskip\noindent
We will next isolate the necessary condition identified in Proposition \ref{pro:strongdiscreteorbit} into a property. 

\par\bigskip\noindent
\begin{defin}
Let $f:X\to X$ be a map and $A\subset X$. The $f$-orbit of $A$ is strongly discrete if there exists an open neighborhood $U$ of $A$ such that $\{f^n(U):n\in \mathbb  Z\}$ is a discrete collection and $f^n(U)\cap f^m(U)=\emptyset$ for distinct $n,m$. The $f$-orbit of $x\in X$ is strongly discrete if the $f$-orbit of $\{x\}$ is strongly discrete.
\end{defin}

\par\bigskip\noindent
In this note we will present partial results addressing Problem \ref{problem:main}. At the end of our study we will identify a few questions that may have a good chance for an affirmative resolution.

\par\bigskip\noindent
In notation and terminology we will follow \cite{Eng}. In particular, if $\prec$ is an order on $L$ and $A,B\subset L$, by $A\prec B$ we denote the  fact that $a\prec b$ for any $a\in A$ and $b\in B$.

\par\bigskip

\section{Study}\label{S:study}

\par\bigskip\noindent
One may wonder if our introduction of the concepts of strongly discrete orbits is really necessary. Can we use the requirement of being "period-point free" instead?  The next example shows that a periodic-point free  automohomeomorphism even on a nice space  need not have strongly discrete orbits.

\par\bigskip\noindent
\begin{ex}
There exist a periodic-point free autohomeomorphism $f$ of the space of rationals $\mathbb Q$ and a point $q\in \mathbb Q$ such that the $f$-orbit of $q$ is not strongly discrete. 
\end{ex}
\begin{proof}
Example  \cite[Example 2.5 ]{Buz} provides a construction of a fixed point autohomeomorphism $f$ on the rationals that
satisfies the hypothesis  of Lemma \cite[Lemma 2.4]{Buz}. For convenience, the cited  hypotheses is copied next:

\par\bigskip\noindent
\underline {Hypothesis of Lemma \cite[Lemma 2.4]{Buz}}: {\it  "Suppose $f:\mathbb Q\to \mathbb Q$ is not an identity map and $p\in \mathbb Q$ satisfy the following property:

\par\smallskip\noindent
(*) $\forall n>0\exists m>0$ such that $f^{m+1}((p-1/n,p+1/n)_\mathbb Q)$ meets $f^{-m}((p-1/n,p+1/n)_\mathbb Q)$."}

\par\bigskip\noindent
Clearly an $f$ that satisfies the above hypothesis fails  having a strong $f$-orbit at $p$.
\end{proof}

\par\bigskip\noindent
For our next affirmative result we need  a technical statement that incorporates our general strategy for showing that a map is potentially monotonic.

\par\bigskip\noindent
\begin{lem}\label{lem:general}
Let $L$ be a GO-space and $f:L\to L$ an autohomeomorphism. Suppose that  $\mathcal O$ is a collection of clopen subsets of $L$ with the following properties:
\begin{enumerate}
	\item The $f$-orbit of each $O\in \mathcal O$ is strongly discrete.
	\item $f^n(O)\cap f^m(O')=\emptyset$ for distinct $O,O'\in \mathcal O$ and $n,m\in \mathbb Z$.
	\item $\{f^n(O): n\in \mathbb Z, O\in \mathcal O\}$ is a cover of $L$ .
\end{enumerate}
Then there exists a GO-compatible  order $\prec$ on $L$ with respect to which $f$ is strictly increasing.
\end{lem}
\begin{proof} By $<$ we denote some ordering with respect to which $L$ is a generalized ordered space. Enumerate elements of $\mathcal O$ as $\{O_\alpha:\alpha<|\mathcal O|\}$. We will define $\prec$ in three stages.

\par\medskip\noindent
\begin{description}
	\item[\it Stage 1] For each  $O\in \mathcal O$ and $n\in \omega\setminus \{0\}$, define $\prec$ on $f^n(O)$  and $f^{-n}(O)$ recursively as follows:

\begin{description}
	\item[\rm Step 0] Put $\prec|_O = <|_O$.

	\item[\rm Assumption] Assume that $\prec$ is defined on $f^k(O)$ and on $f^{-k}(O)$ for all $k=0,1,...,n-1$.

	\item[\rm Step $n$] If $x,y\in f^n(O)$, put $x\prec y$ if and only if $f^{-1}(x) \prec f^{-1}(y)$. This is well defined since $f^{-1}(x), f^{-1}(y)$ are in $f^{n-1}(O)$ and $\prec$ is defined on $f^{n-1}(O)$ by assumption. Similarly,  if $x,y\in f^{-n}(O)$, put $x\prec y$ if and only if $f(x) \prec f(y)$.
\end{description}

	\item[\it Stage 2] For any $O\in \mathcal O$ and any $n,m\in \mathbb Z$ such that $n<m$, put $f^n(O) \prec f^m(O)$.
	\item[\it Stage 3] For any $\alpha<\beta<|\mathcal O|$ and any $n,m\in \mathbb Z$, put $f^n(O_\alpha)\prec f^m(O_\beta)$.
\end{description}
The next two claims show that $\prec$ is as desired.
\par\smallskip\noindent
{\it Claim 1. $\prec$ is compatible with the GO-topology of $L$.}
\par\smallskip\noindent
{\it Proof of Claim}. To prove the claim, for each $O\in \mathcal O$, let $\mathcal T_O$ be the collection of all  $<$-convex open subsets of $L$  that are subsets of $O$. Since $<$ coincides with $\prec$ one every $O\in \mathcal O$, we conclude that every element in $\mathcal T_O$ is $\prec$-convex. By  the constructions at Stage 1, $f^n(O)$ is $\prec$-convex. Since $f$ is an autohomeomorphism, the collection $\{f^n(I): I\in\mathcal T_O, O\in \mathcal O\}$ is a basis for the topology of $L$ and consists of open $\prec$-convex sets. The claim is proved. 
\par\smallskip\noindent
{\it Claim 2. $f$ is increasing with respect to $\prec$.}
\par\smallskip\noindent
{\it Proof of Claim}. Pick distinct $x$ and $y$. If $x,y\in \bigcup_nf^n(O)$ for some  $O\in \mathcal O$, then apply Stages 1 and 2. Otherwise, apply Stage 3.
\end{proof}

\par\bigskip\noindent
{\bf Remark to Lemma \ref{lem:general}.} {\it Note that if  $\langle L, <\rangle$ is a LOTS and each $O$ in the argument of the lemma has both extremities or each $O$ has neither extremity, then $\langle L, \prec\rangle$ is a LOTS too.}

\par\bigskip\noindent
The converse of Lemma \ref{lem:general} for fixed-point free autohomeomorphisms on zero-dimensional GO-spaces holds too (Lemma \ref{lem:generalconverse}). To prove the converse, we need the following  quite technical statement. Recall that given a continuous  self-map $f:X\to X$, a closed set $A\subset X$ is an {\it $f$-color} if $A\cap f(A)=\emptyset$. For a review of major results on colors of continuous maps, we refer the reader to \cite{vM}.

\par\bigskip\noindent
\begin{pro}\label{pro:maximalcolor}
Let $L$ be a zero-dimensional GO-space, $f:L\to L$  a fixed-point free monotonic homeomorphism, and $x\in L$. Then there exists a maximal convex clopen set $I\subset L$ containing $x$ such that the following hold:
\begin{enumerate}
	\item $\bigcup_{n\in \mathbb Z} f^n(I)$ is clopen and convex.
	\item $f^n(I)\cap f^m(I)=\emptyset$ for any distinct integers $n$ and $m$. 
	\item $f^n(I)$ is a maximal clopen convex $f$-color for any $n\in \mathbb Z$.
\end{enumerate}
\end{pro}
\begin{proof}
We may assume that $f$ is strictly increasing. If $L$ is locally compact, by zero-dimensionality there exists $a\in L$ such that $x<a$ and $\{y\in L: y<a\}$ is clopen and non-empty. Put $I = [a,f(a))$. Let us show that $I$ is as desired. By monotonicity,
$\{f^n(I):n\in \mathbb Z\}= \{..., [f^{-1}(a),a), [a,f(a)), [f(a),f^2(a)),. ..\}$. Enlarging any interval in this sequence would make that interval meet its image. Therefore, (3) is met. Visual inspection of the sequence is a convincing evidence  that
the union $\bigcup_n f^n(I)$ is convex. The union  is also open as the union of open sets. Since $f$ is fixed-point free, $f^n(I)$'s form a discrete collection, and hence, the union is closed. By our choice, $f^{n+1}(I) = [f^{n+1}(a), f^{n+2}(a))$, which guarantees that (2) is met. 
 
We now assume that $L$ is not locally compact. Let $dL$ be the smallest ordered compactification of $\langle L, <\rangle$. Since $f$ is a monotonic autohomeomorphism, $f$ has a unique continuous extension $\tilde f:dL\to dL$. Let $F\subset dL$ be the set of all fixed points of $\tilde f$. Since $f$ is fixed-point free, $F$ is a subset of  $dL\setminus L $. Since $L$ is not locally compact, there exists $a \in (dL\setminus L)\setminus F$. Put $I = [a, \tilde f(a)]_{dL}\cap L$. Clearly, $I$ is as desired.
\end{proof}

\par\bigskip\noindent
\begin{lem}\label{lem:generalconverse}
Let $L$ be a zero-dimensional GO-space and let  $f:L\to L$ be a fixed-point free monotonic autohomeomorphism. Then there exists a collection $\mathcal O$  of convex clopen subsets of $L$ with the following properties:
\begin{enumerate}
	\item The $f$-orbit of each $O\in \mathcal O$ is strongly discrete.
	\item $f^n(O)\cap f^m(O')=\emptyset$ for distinct $O,O'\in \mathcal O$ and $n,m\in \mathbb Z$.
	\item $\{f^n(O): n\in \mathbb Z, O\in \mathcal O\}$ is a cover of $L$ .
\end{enumerate}
\end{lem}
\begin{proof} Without loss of generality, we may assume that $f$ is strictly increasing.
We will construct $\mathcal O=\{O_\alpha\}_\alpha$ recursively. Assume that $O_\beta$ is constructed for each $\beta<\alpha$ and the following properties  hold:
\begin{description}
	\item[\rm P1] $\bigcup_{n\in \mathbb Z} f^n(O_\beta)$ is clopen and convex.
	\item[\rm P2]  $f^n(O_\beta)\cap f^m(O_\beta)=\emptyset$ for any distinct integers $n$ and $m$. 
	\item[\rm P3] $f^n(O_\beta)$ is a maximal clopen convex $f$-color for any $n\in \mathbb Z$.
\end{description}
Note that P1 and P2 imply the following:
\begin{description}
	\item[\rm P4] The $f$-orbit of $O_\beta$ is strongly discrete.
\end{description}
\par\medskip\noindent
\underline {\it Construction of $O_\alpha$}: Let $L_\alpha = L\setminus \bigcup \{f^n(O_\beta):\beta<\alpha, n\in \mathbb Z\}$. If $L_\alpha$ is empty, then the recursive construction is complete and $\mathcal O = \{O_\beta:\beta<\alpha\}$. Otherwise,  we have $f^{-1}(f(L_\alpha)) = L_\alpha$. Let us show that $L_\alpha$ is clopen in $L$. Firstly, it is closed as the complement of the union of open sets. To show that it is open, fix $x\in L_\alpha$. Let $I$ be as in  Proposition  \ref{pro:maximalcolor} for given $x, f, L$. If $x$ were a limit point for $L\setminus L_\alpha$, then it would have contained some $f^n(O_\beta )$  for $\beta<\alpha$ and $n\in \mathbb Z$, which contradicts property P3.   Hence, $I$ is an open neighborhood of $x$ contained in $L_\alpha$. Since properties (1)-(3) of $I$ in the conclusion of Proposition \ref{pro:maximalcolor} coincide with the properties P1-P3, we can put $O_\alpha = I$.

\par\medskip\noindent
The family $\mathcal O=\{O_\alpha\}_\alpha$ is as desired by construction.
\end{proof}

\par\bigskip\noindent
Lemmas \ref{lem:general} and \ref{lem:generalconverse} form the following criterion.

\par\bigskip\noindent
\begin{thm}\label{thm:criterion}
Let $f$ be a fixed-point free autohomeomorphism on a zero-dimensional GO-space $X$. Then $f$ is potentially monotonic if and only if there exists a collection $\mathcal O$  of convex clopen subsets of $L$ with the following properties:
\begin{enumerate}
	\item The $f$-orbit of each $O\in \mathcal O$ is strongly discrete.
	\item $f^n(O)\cap f^m(O')=\emptyset$ for distinct $O,O'\in \mathcal O$ and $n,m\in \mathbb Z$.
	\item $\{f^n(O): n\in \mathbb Z, O\in \mathcal O\}$ is a cover of $L$ .
\end{enumerate}
\end{thm}

\par\bigskip\noindent
We next put one part (Lemma \ref{lem:general}) of the above criterion to a good use.

\par\bigskip\noindent
\begin{thm}\label{thm:main}
Let $X$ be a zero-dimensional subspace of the reals and let $f:X\to X$ be an autohomeomorphism with strongly discrete orbits at all points. Then there exists a GO-compatible order $\prec$ on $X$  such that $f$  is $\prec$-monotonic.
\end{thm}
\begin{proof}
To prove the statement, we will construct a collection $\mathcal O$ as in the hypothesis of Lemma \ref{lem:general}. Fix a countable cover $\mathcal F = \{F_n:n\in \omega\}$ of $X$ so that each $F_i$ is clopen and has strongly discrete $f$-orbit.

\begin{description}
	\item[\it Step $0$]  Put $O_0=F_0$.
	\item[\it Assumption] Assume that $O_k$ is defined for $k<n$, clopen, and has strongly discrete $f$-orbit. In addition, assume that $\bigcup_{m\in\mathbb Z}f^m(O_i)$ misses $\bigcup_{m\in\mathbb Z}f^m(O_j)$, whenever $i\not = j$ and $i,j<n$.
	\item[\it Step $n$] Let $i_n$ be the smallest index such that $F_{i_n}$ is not covered by $\{f^m(O_i): i<n, m\in \mathbb Z\}$. Put $O_n = F_{i_n}\setminus \cup\{f^m(O_i): i<n, m\in \mathbb Z\}$. 
\end{description}
\par\bigskip\noindent
Construction is complete. The collection  $\mathcal O = \{O_n: n\in \omega\}$ has properties (1) and (2) in the hypothesys of Lemma \ref{lem:general} by construction.
To show (3), that is, the equality $X= \cup\{f^m(O_i): i\in \omega, m\in \mathbb Z\}$, fix any $x\in X$. Since $\mathcal F$ is a cover of $X$, there exists $n$ such that $x\in F_n$. If $x$ is not in $f^m(O_i)$ for some $i<n$ and $m\in \omega$, then $F_n$ is the first element in $\mathcal F$ that meets the construction requirements at step $n$. Therefore, $x\in O_n$.
\end{proof}

\par\bigskip\noindent
\begin{cor}\label{cor:potentiallymonotonicZ}
Every periodic-point free bijection on $\mathbb Z$ is potentially monotonic.
\end{cor}

\par\bigskip\noindent
In contrast with Corollary \ref{cor:potentiallymonotonicZ}, we next observe that  not every periodic-point free bijection on $\mathbb Z$ is topologically equivalent to a monotonic map.
\par\bigskip\noindent
\begin{ex}\label{ex:notequivalent}
There exists a periodic-point free  bijection on $\mathbb Z$ that is not topologically equivalent to a monotonic map.
\end{ex}
\begin{proof}
First observe that every monotonic  bijection on $\mathbb Z$ is a shift. Therefore, any bijection on $\mathbb Z$ that is topologically equivalent to a monotonic map is also  topologically equivalent to a shift. It is observed in \cite[Example 1.2]{BW} that if a bijection $f$ on $\mathbb Z$ has infinitely many points with mutually disjoint orbits, then such a map is not topologically equivalent to a shift. Thus, any such fixed-point free  map is an example of a potentially monotonic map on $\mathbb Z$ that is not topologically equivalent to a monotonic map.
\end{proof}

\par\bigskip\noindent
\begin{rem}\label{rem:rem1}
Corollary \ref{cor:potentiallymonotonicZ} and Example \ref{ex:notequivalent} imply that the property of being potentially monotonic does not imply the property of being topologically equivalent to a monotonic map (with respect to the existing order).
\end{rem}

\par\bigskip\noindent
We can strengthen Corollary \ref{cor:potentiallymonotonicZ} as follows.
\par\bigskip\noindent
\begin{thm}
Let  $f$ be a periodic-point free bijection on  $\mathbb Z$. Then there exist an ordering $\prec$ and a binary operation $\oplus$  on $Z$ such that $\mathbb Z'=\langle \mathbb Z, \oplus, \prec\rangle$ is a discrete ordered topological group and $f$ is a shift in $\mathbb Z'$.
\end{thm}
\begin{proof}
Let $M$ be a minimal subset of $\mathbb Z$ with respect to the property that the $f$-orbit of $M$ covers $\mathbb Z$. 

If $|M|=n$, enumerate the elements of $M$ by $\mathbb Z_n$. Clearly, $\mathbb Z_n\times_l \mathbb Z$ is an ordered discrete  topological group with the component-wise addition. Define a bijection $h: \mathbb Z\to \mathbb Z_n\times_l \mathbb Z$ by letting $g(f^k(n_i)) = (i, k)$. Since any element of $\mathbb Z$ is in the $f$-orbit of exactly one element of $M$, the correspondence is well-defined and is a bijection. Since $g$ is a homeomorphism, we will next abuse notation and will identify $f^k(x_i)$ with $(i,k)$. Let us apply $f$ to $(i,k)$. We have  $f(f^k(x_i)) = f^{k+1}(x_i)$, and the latter is identified with $(i, k+1)$. Therefore, $f$ is a shift by $(0,1)$ in $\mathbb Z'$.

If $M$ is infinite, enumerate its elements by integers as $M=\{n_i:i\in \mathbb Z\}$.  Define  $h: \mathbb Z\to \mathbb Z\times_l \mathbb Z$ by letting $g(f^k(n_i)) = (i, k)$. Argument similar to the $\mathbb Z_n$ case shows that the ordering on $\mathbb Z$ induced by $h$ is as desired.
\end{proof}

\par\bigskip\noindent
Note that the above statement does not hold for continuous periodic-point free bijections on the rationals. Indeed, as shown in \cite[Example 2.5]{Buz} there exists a continuous periodic-point free bijection on $\mathbb Q$ with a point with non-strongly discrete fiber. The mentioned example \cite[Example 2.5]{Buz} is constructed to satisfy the hypothesis of \cite[Lemma 2.4]{Buz}, which is a stronger case of not having discrete fibers.   Nonetheless, the following takes place.

\par\bigskip\noindent
\begin{thm}\label{thm:rationals}
A fixed-point free autohomeomorphism $f:\mathbb  Q\to \mathbb Q$ is potentially monotonic if and only if $f$ is topologically equivalent to a shift.
\end{thm}
\begin{proof}
($\Rightarrow$) Since $f$ is potentially monotonic, there exists a collection $\mathcal O$ as in the conclusion of Lemma \ref{lem:generalconverse}. The argument of Theorem 2.3 in \cite{Buz} shows that that $f$ with such a collection is topologically equivalent to a non-trivial shift. 

\par\medskip\noindent
($\Leftarrow$) It is proved in \cite[Theorem 2.8]{BW} that
a periodic-point free homeomorphism $h$ on $\mathbb Q$  is topologically equivalent to a shift if and only if one can introduce a group operation $\oplus$ on  $\mathbb Q$ compatible with the topology of $\mathbb Q$ so that the topological group $\langle \mathbb Q , \oplus\rangle$ is continuously isomorphic to $\mathbb Q$ and $h$  is a shift with respect to new operation. Clearly such an $\langle \mathbb Q , \oplus\rangle$ is an ordered topological group,  and hence,  any shift is monotonic.  Therefore, $f$ is potentially monotonic.
\end{proof}

\par\bigskip\noindent
Recall that given a continuous selfmap $f:X\to X$, {\it the chromatic number of $f$} is the least number of $f$-colors needed to cover $X$. 

\par\bigskip\noindent
\begin{thm}\label{thm:strictlymonotoniciff2colored}
Let $f$ be a fixed-point free  autohomeomorphism on a zero-dimensional GO-space $L$. If $f$ is potentially monotonic, then the chromatic number of $f$ is $2$.
\end{thm}
\begin{proof}
($\Rightarrow$) Since the chromatic number of $f$ is a purely topological property not attached to an order, we may assume that $f$ is strictly monotonic. Let $\mathcal O$ be as in the conclusion of  \ref{lem:generalconverse} for the given $f$ and $L$. Put $A=\cup\{f^n(O): n\  is\ an \ even\ integer,\  O\in \mathcal O\}$ and $B=\cup\{f^n( O): n\ is\ an\  odd\ integer,\ n\in \mathcal O\}$. Clearly, $\{A,B\}$ is cover of $L$ by colors.
\end{proof}

\par\bigskip\noindent
Theorem \ref{thm:strictlymonotoniciff2colored} and Remark \ref{rem:rem1} prompt the following question.

\par\bigskip\noindent
\begin{que}
Let $f$ be a  periodic point free  homeomorphism on a zero-dimensional GO-space $L$. Let the chromatic number of $f$ be 2. Is $f$ potentially monotonic?
\end{que}

\par\bigskip\noindent
Theorem \ref{thm:main} prompts the following question.

\par\bigskip\noindent
\begin{que}\label{que:question1}
Let $X$ be a GO-space and let $f:X\to X$ be an autohomeomorphism with strongly discrete orbits at all points.Is $f$ potentially monotonic? What if $X$ is hereditarily paracompact?
\end{que}

\par

\end{document}